\documentclass[reqno]{amsart}
\usepackage[english]{babel}

\usepackage{amsfonts}
\usepackage{amssymb}
\usepackage{color}

\newtheorem{thm}{Theorem}[section]

\newtheorem{lem}[thm]{Lemma}

\theoremstyle{definition}
\newtheorem{defn}{Definition}[section]

\newtheorem*{problem*}{Problem}

\theoremstyle{remark}

\numberwithin{equation}{section}

\DeclareMathOperator*{\res}{res}

\begin{document}

\title[Multidimensional algebraic interpolations]{Multidimensional algebraic interpolations
}

\author{M.E. Durakov}

\address[Matvey Durakov]
{Siberian Federal University
                                                 \\
         pr. Svobodnyi 79
                                                 \\
         660041 Krasnoyarsk
                                                 \\
         Russia}
\email{durakov\_m\_1997@mail.ru}

\author{E.D. Leinartas}

\address[Evgeniy Leinartas]
{Siberian Federal University
                                                 \\
         pr. Svobodnyi 79
                                                 \\
         660041 Krasnoyarsk
                                                 \\
         Russia}
\email{elyaynartas@sfu-kras.ru}

\author{A.K. Tsikh}

\address[August Tsikh]
{Siberian Federal University
                                                 \\
         pr. Svobodnyi 79
                                                 \\
         660041 Krasnoyarsk
                                                 \\
         Russia}
\email{atsikh@sfu-kras.ru}

\subjclass {41A05, 32A27}
\keywords{Grothendieck residue, Interpolation, Local algebra}

\begin{abstract}
The Hermite interpolation formulas are based on the interpretation of interpolation nodes as roots of suitable polynomials. Therefore, such formulas belong to the class of algebraic interpolations. The article considers a multidimensional variant of Hermite interpolation, presents a class of algebraic systems of equations for which the Hermite interpolation polynomial is represented by an explicit formula. The theory of multidimensional residues is used as the main tool.
\end{abstract}

\maketitle

\section*{Introduction}
\label{s.Int}

The Lagrange and Hermite interpolation formulas are based on the interpretation of interpolation nodes as roots of suitable polynomials. Therefore, such formulas belong to the class of algebraic interpolations. The ideology of algebraic interpolations has recently received close attention from the point of view of interpolation theory for functions in several variables (see, for example \cite{Alpay}-\cite{newAY}).

The article deals with the interpolation of functions in a complex space for which the sets of nodes are 0-dimensional analytic sets. In the one-dimensional case, these are Hermite interpolations. As the main apparatus, we use the theory of multidimensional residues.

Basic classical (standard) interpolations are understood as Lagrange, Hermite, Newton, etc. interpolations. Let's consider the first two of them.

\begin{problem*}[Lagrange]
Given a set of distinct points $\{w_j\}_{j=1}^m \subset \mathbb{C}$ and the values $c_j \in \mathbb{C}$, find the polynomial $f(z)$ of degree $m-1$ with the property $$f(w_j)=c_j,\quad j = 1, \ldots, m.$$
\end{problem*}

Note that the interpolation polynomial $f$ is defined in terms of the polynomial $p(z)=(z-w_1)\cdot\ldots\cdot(z-w_m)$ by the formula: $$f(z)=p(z)\sum\limits_{j=1}^m\frac{c_j}{z-w_j}\res\limits_{w_j}\left(\frac{1}{p}\right).$$
Thus, specifying the interpolation nodes in the form of the null set of the polynomial $p$ provides tools for constructing an interpolation polynomial by using residues. More general is the following

\begin{problem*}[Hermite]
Let $\{w_j\}_{j=1}^m\subset\mathbb{C}$ be a set of pairwise distinct points and the following values are given $$c_{j, \ell} \in \mathbb{C}, \quad\text{where } j = 1, \ldots, m,\quad\ell = 0, \ldots, \mu_j\,-\,1.$$ It is necessary to find a polynomial $f(z)$ of minimal degree, which at points $w_j$ has the given values of derivatives up to orders of $\mu_j-1$ inclusive, that is
\begin{equation}\label{asd}f^{(\ell)}(w_j)=c_{j, \ell}, \quad j = 1, \ldots, m, \quad \ell = 0, \ldots, \mu_j-1.\end{equation}
\end{problem*}

In the Hermite interpolation problem, it is advisable to enumerate the set of points $w_j$ taking into account their multiplicities, thereby considering the set $\{w_j\}_{j=1}^m$ as an algebraic set $p^{-1}(0)$, where \begin{equation}\label{polynomial} p(z) = (z-w_1)^{\mu_1}\cdot\ldots\cdot(z-w_m)^{\mu_m}.\end{equation}
The Hermite interpolation polynomial can be represented as

\begin{equation*}
\sum\limits_{j=1}^m  \frac{p(z)}{ (z-w_j)^{\mu_j}} \sum\limits_{\ell=0}^{\mu_j-1}\frac{c_{j, \ell}}{\ell !}\left(\sum\limits_{s=0}^{\mu_j-\ell-1}  (z-w_j)^{\ell+s} \res\limits_{w_j}\left(\frac{(z-w_j)^{\mu_j-1-s}}{p}\right)   \right),
\end{equation*}
that is, again with the participation of residues.

Our goal is to construct a similar variant for multidimensional interpolation. First, we will explain the main points of constructing an interpolation polynomial in the case of Lagrange interpolation. Let $\boldsymbol{p} = (p_1, \ldots, p_n)$ be a system of $n$ polynomials of $n$ variables $z= (z_1, \ldots, z_n)$ with a finite set of roots $\boldsymbol{p}^{-1}(0)\subset\mathbb{C}^n$. Assume that all roots are simple, this means that the Jacobian $J_{\boldsymbol{p}}$ of the system (mapping) $\boldsymbol{p}$ is not equal to zero at the roots:
$$J_{\boldsymbol{p}}(w)\neq 0, \quad \forall w\in \boldsymbol{p}^{-1}(0).$$

\begin{problem*}[Multidimensional Lagrange Interpolation]
For the given values $\{c_{w}\},$
$w\in\boldsymbol{p}^{-1}(0)$, find the polynomial $f(z)$ with the condition
\begin{equation}\label{one}
f(w)=c_{w}, \qquad w\in \boldsymbol{p}^{-1}(0).
\end{equation}
\end{problem*}

In the one-dimensional case, we can rewrite the Lagrange polynomial in the following form:
$$f(z) = \sum\limits_{w\in p^{-1}(0)} c_{w}\frac{p(z)}{z-w}\res\limits_{w}\left(\frac{1}{p}\right).$$
Therefore, we need to somehow interpret the division of the ideal $\langle p\rangle$, generated by the polynomial $p$ in the ring $\mathbb{C}[z]$, by its primary component $\langle z-w\rangle$ at the root of $w$ see \cite{VDV}. The polynomial $H(z, w)$ from the following decomposition is taken as such an interpretation
$$p(z)-p(\zeta) = (z - \zeta)H(z, \zeta).$$
So the Lagrange formula takes the form:
$$f(z) = \sum\limits_{w\in p^{-1}(0)} c_{w} H(z, w) \res\limits_{w}\left(\frac{1}{p}\right).$$

In the multidimensional case there is a system of identities
$$p_j(z)-p_j(\zeta) = \sum\limits_{k=1}^n (z_k-\zeta_k)h_{jk}(z, \zeta), \qquad j = 1,\ldots, n$$
with a polynomial matrix $H(z, \zeta) = (h_{jk}(z,\zeta)).$
It is obvious that on the diagonal $\zeta = z$ the matrix $(h_{jk})$ coincides with the Jacobi matrix of the mapping $\boldsymbol{p}.$ At the roots of $w\in\boldsymbol{p}^{-1}(0)$, the determinant $\det{H(z,w)}$ has the following properties:
\begin{itemize}
\item$\det{H(w,w)}$ is equal to the Jacobian $J_{\boldsymbol{p}}(w)$ of the mapping $\boldsymbol{p}$;
\item $\det{H(w, w')} = 0$ for each pair of different roots $w, w' \in \boldsymbol{p}^{-1}(0).$
\end{itemize}
For the second property, see Lemma 4.1 in section 4.
Local residue (its exact definition is given in section 2) at the simple root $w\in\boldsymbol{p}^{-1}(0)$, that is at the root where $J_{\boldsymbol{p}}(w) \neq 0$, is calculated by the formula (\cite{Tsikh}, section 5.5):
$$\res\limits_{w} \left(\frac{h}{\boldsymbol{p}^I}\right) = \frac{h(w)}{J_{\boldsymbol{p}}(w)}.$$

From these properties we get the following result:

\begin{thm}
The polynomial
$$f(z) = \sum_{w \in \boldsymbol{p}^{-1}(0)} c_{w} \det{H(z, w)} \res\limits_{w}\left(\frac{1}{\boldsymbol{p}^I}\right),$$
where $\res\limits_{w}\left(\frac{1}{\boldsymbol{p}^I}\right)$ is a local residue, solves the problem (\ref{one})
\end{thm}

\section{Multidimensional variant of Hermite interpolation}
\label{s.1}

Let's consider a multidimensional analogue of the Hermite problem.\linebreak  Let $\boldsymbol{p} = (p_1, \ldots, p_n)$ be a sequence of $n$ polynomials of $n$ variables $z= (z_1, \ldots, z_n)$ with a finite number of common roots $\boldsymbol{p}^{-1}(0)$. Obviously, $\boldsymbol{p}$ can be interpreted as mapping $\boldsymbol{p}: \mathbb{C}^n\to \mathbb{C}^n$. At the same time, $\boldsymbol{p}$ is associated with the ideals $\langle\boldsymbol{p}\rangle$ and $\langle \boldsymbol{p} \rangle_w$ generated by the sequence $p_j$ both in the ring $\mathbb{C}[z_1, \ldots, z_n],$ and in the local rings $\mathcal{O}_w$ of germs of holomorphic functions at the roots $w\in\boldsymbol{p}^{-1}(0)$. The quotient ring $\mathcal{O}_{w}/\langle\boldsymbol{p}\rangle_w$ has the structure of a vector space and is called a local algebra at the point $w$. Denote by $B_w$ the monomial basis of the local algebra $\mathcal{O}_{w}/\langle\boldsymbol{p}\rangle_w$ at the point $w\in \boldsymbol{p}^{-1}(0)$. The dimension of this algebra coincides with the multiplicity of the mapping $\boldsymbol{p}$ at the root of $w$ (see \cite{Tsikh}, section 19.4). The set of exponents $\ell$ in basic monomials $(z-w)^\ell\in B_w$ we denote as $A_w$.

\begin{problem*}[Multidimensional Hermite interpolation]
For the given values $\{c_{w, \ell}\}$, $w\in\boldsymbol{p}^{-1}(0); \ell \in A_{w}$, find the polynomial $f(z)$ with the condition
$$\frac{\partial^{|\ell|} f}{\partial z^\ell}(w)=c_{w, \ell}, \qquad w\in \boldsymbol{p}^{-1}(0), \qquad \ell \in A_{w},$$
where $\ell = (\ell_1, \ldots, \ell_n)$  and  $|\ell| = \ell_1+\ldots+\ell_n$.
\end{problem*}

Here we will consider a multidimensional version of the above problem in the case when the Newton polyhedra of the supports $A_{w}$ are essentially parallelepipeds $$\{0\leq\ell_1\leq d_1-1\}\times\ldots\times \{0\leq\ell_n\leq d_n-1\}.$$

To formulate the main theorem, we need the following definition of the\linebreak Grothendieck residue.

\begin{defn}{(see \cite{GH} or \cite{Tsikh})}
Let $h, p_1, \ldots,p_n$ be the holomorphic functions in the neighborhood $U_a$ of the point $a\in\mathbb{C}^n$, such that the mapping $\boldsymbol{p}=(p_1,\ldots,p_n)$ has an isolated zero at the point $a$: $\boldsymbol{p}^{-1}(0)\cap U_a = \{a\}$. By local residue or by Grothendieck residue of the meromorphic form $\omega=h\,dz/(p_1\cdots p_n)$ at the point $a$ we will call the integral $$\res\limits_{a}\omega=(2\pi i)^{-n}\int\limits_{\Gamma_a}\frac{h(z)\,dz_1\wedge\ldots\wedge dz_n}{p_1(z)\cdots p_n(z)},$$
over the cycle $$\Gamma_a=\{z\in U_a: |p_j(z)|=\varepsilon, j = 1, \ldots, n\},$$ where $\varepsilon$ is quite small, such that $\Gamma_a \Subset U_a;$ the orientation of the cycle $\Gamma_a$ is determined by the condition $$d(\arg p_1)\wedge\ldots\wedge d(\arg p_n)\geq 0.$$
\end{defn}

For the local residue, we will also use the notation $$\res\limits_a\left(\frac{h}{\boldsymbol{p}^I}\right),$$
assuming that $I = (1, \ldots, 1) \in \mathbb{Z}^n_+,$ and therefore the product $p_1\ldots p_n$ in the denominator of the integral is nothing but a monomial $\boldsymbol{p}^I = p_1\ldots p_n$. The main properties of local residues and methods of their calculations are given in \cite{GH}, \cite{Tsikh}, \cite{TY} and \cite{Tajima}.


Note that in the one-dimensional Hermite problem, in each root $w = w_j$ of a polynomial (\ref{polynomial}), it is required to restore $\mu_j$ values of derivatives (\ref{asd}), where $\mu_j$ is the multiplicity of the root $w_j$. In a multidimensional problem, the situation is the same. Recall (see \cite{Tsikh} or \cite{Milnor}) that the multiplicity of the isolated zero $w\in\mathbb{C}^n$ of the germ of the holomorphic map $\boldsymbol{p}:(\mathbb{C}^n, w)\to(\mathbb{C}^n, 0)$ is defined as the limit $$\mu_w = \overline{\lim\limits_{\xi\to w}} \# \{U_w \cap \boldsymbol{p}^{-1}(\xi)\},$$
where $\#$ is the power sign (cardinal) of the set, and $U_w$ is the neighborhood of a point, in the closure of which there are no roots of $\boldsymbol{p}$ other than $w$. The multiplicity $\mu_w$ coincides with the degree of mapping $z\to \boldsymbol{p}(z) / ||\boldsymbol{p}(z)||$ from the boundary $\partial U_w$ to the unit sphere (see \cite{Milnor}, Lemma B2).

In the one-dimensional case $(n=1)$, there are two equivalent definitions of multiplicity:
\begin{enumerate}
\item $z=w$ is a zero of multiplicity $d$ of the function $p(z)$ if function could be represented as
$$p(z) = (z-w)^{d} \varphi(z), \qquad \varphi(w) \neq 0;$$
\item $z=w$ is a zero of multiplicity $d$ of the function $p(z)$ if at the point $w$ the function $p$ and its derivatives up to the order of $d$ have the property
$$p(w) = p'(w) = \ldots = p^{(d-1)}(w) = 0, p^d(w) \neq 0.$$
\end{enumerate}

In a multidimensional situation $(n\geq 2)$ the concepts of multiplicity and order of zero differ: the first of them applies to the maps $\mathbb{C}^n\to \mathbb{C}^n$, and the second to the maps $\mathbb{C}^n\to \mathbb{C}^1$ (i.e. to functions). The concept of multiplicity is quite complex, there is no universal formula for its calculation. In order to solve the Hermite problem constructively (in the form of an explicit formula) we will consider the class introduced in the article \cite{LeTs} of polynomial maps $\boldsymbol{p}: \mathbb{C}^n\to\mathbb{C}^n$ which in each root $w\in\boldsymbol{p}^{-1}(0)$ have properties similar to the one-dimensional definition (ii). To describe these properties, the inequality $a\leq b$ of vectors $a, b\in\mathbb{Z}^n$ will be understood as the system of inequalities $a_i\leq b_i$ for coordinates of these vectors. Denote by $I$ the vector $(1, \ldots, 1) \in\mathbb{Z}^n$. The order of zero $w\in\boldsymbol{p}^{-1}(0)$ is the vector $d_w = (d_1, \ldots, d_n) \in\mathbb{Z}_{+}^n$ such that

\begin{equation}\label{first}
\frac{\partial^{|\ell|} p_i}{\partial z^\ell}(w) = 0, \  0\leq \ell \leq d_w -I,
\end{equation}

\begin{equation}\label{second}
\det\left|\left|\frac{\partial^{d_{k}} p_i}{\partial z_k^{d_{k}}}(w)\right|\right| \neq 0\quad (\text{ here } i, k = 1,\ldots, n).
\end{equation}

\begin{thm}\label{mainth}
Suppose that for each root $w\in\boldsymbol{p}^{-1}(0)$ there is such a vector  \linebreak $d_w = (d_1, \ldots, d_n) \in\mathbb{Z}_{+}^n$, that the properties (\ref{first}) and (\ref{second}) are being fulfilled.
Then the polynomial
\begin{equation}\label{int}
f(z) = \sum_{w \in \boldsymbol{p}^{-1}(0)} \det {H_w(z)} \left[ \sum_{\substack{\ell \leq d_w-I \\ k \leq d_w-I-\ell} } \frac{c_{w, \ell}}{\ell!}  (z-w)^{\ell+k}\res\limits_{w}\left(\frac{(z-w)^{d_w -I-k}}{\boldsymbol{p}^I}\right) \right]
\end{equation}
solves the multidimensional Hermite problem with nodes $w\in\boldsymbol{p}^{-1}(0)$, where \linebreak $H_w(z)~=~||h_{ik}||$ is a matrix from the representation
\begin{equation}\label{decomp}
\begin{pmatrix}
p_1(z)\\
\vdots\\
p_n(z)
\end{pmatrix}=
||h_{ik}||
\begin{pmatrix}
(z_1-w_{1})^{d_{1}}\\
\vdots\\
(z_n-w_{n})^{d_{n}}
\end{pmatrix}.
\end{equation}
\end{thm}

Note that the existence of the representation (\ref{decomp}) will be proved by the lemma \ref{lem} below. We also note that when $w$ in (\ref{int}) is fixed, then the set $\{0\leq\ell\leq d_w-I\}$ of values of $\ell$ in (\ref{int}) means that $\ell$ runs over the parrallelepiped $A_w$ mentioned in statement of Hermite problem. So the cardinal $\# A_w = d_1 \cdot\ldots\cdot d_n$ and it is equal to the multiplisity $\mu_w = \mu_w(p)$ of the root $w$ when $p$ satisfies properties (\ref{first}) and (\ref{second}) (see \cite{LeTs}).

Consider the following example in $\mathbb{C}^2$, which satisfies the conditions of the theorem:
\begin{equation*}
\begin{cases}
p_1(z) = (z_1-1)^2+(z_2-1)^2+\frac{1}{3}(z_1-1)^2 (z_2-1)^2 \\
p_2(z) = (z_1-1)^2-(z_2-1)^2+\frac{2}{3}(z_1-1)^2 (z_2-1)^2.
\end{cases}
\end{equation*}
This mapping has five different roots: one of them $w_0 = (1, 1)$ of multiplicity 4 and four simple roots $w_j = (1\pm\sqrt{6}, 1\pm i\sqrt{2}), j = 1, 2, 3, 4$. The monomial basis of a local algebra at the root $w_0$ consists of four monomials: $1, z_1, z_2, z_1 z_2$. Therefore, we need to find a polynomial $f(z)$ such that:
$$f(w_0) = c_{0,(0,0)}, f'_{z_1}(w_0) = c_{0,(1,0)}, f'_{z_2}(w_0) = c_{0,(0,1)}, f''_{z_1 z_2}(w_0) = c_{0,(1,1)},$$
$$f(w_1) = c_1, f(w_2) = c_2, f(w_3) = c_3, f(w_4) = c_4.$$

Let's find the matrix $H_{w_0}(z)$ from the condition:

\begin{multline*}
\begin{pmatrix}
p_1(z) \\
p_2(z)
\end{pmatrix} = 
\begin{pmatrix}
(z_1-1)^2+(z_2-1)^2+\frac{1}{3}(z_1-1)^2 (z_2-1)^2 \\
(z_1-1)^2-(z_2-1)^2+\frac{2}{3}(z_1-1)^2 (z_2-1)^2
\end{pmatrix}=\\
\begin{pmatrix}
(z_1-1)^2\left(1+\frac{1}{3}(z_2-1)^2\right)+(z_2-1)^2 \\
(z_1-1)^2\left(1+\frac{2}{3}(z_2-1)^2\right)-(z_2-1)^2
\end{pmatrix} = \\
\begin{pmatrix}
1+\frac{1}{3}(z_2-1)^2, &1\\
1+\frac{2}{3}(z_2-1)^2, &-1
\end{pmatrix}
\begin{pmatrix}
(z_1-1)^2 \\
(z_2-1)^2
\end{pmatrix} = H_{w_0}(z)
\begin{pmatrix}
(z_1-1)^2 \\
(z_2-1)^2
\end{pmatrix}.
\end{multline*}
Thus $\det{H_{w_0}(z)} = -2-(z_2-1)^2$. 
For $h = \sum\limits_{|\alpha|\geq 0} h_\alpha(z-w_0)^\alpha\in \mathcal{O}_{w_0}$, by the transformation formula of the local residue for mapping from $\boldsymbol{p}$ to $(H_{w_0}(z))^{-1}\boldsymbol{p}$, we~get (see \cite{Tsikh}, section 5.5)
$${\res\limits_{w_0}}\left(\frac{h}{p_1 p_2}\right) = {\res\limits_{w_0}}\left(\frac{h/ \det H_{w_0}(z)}{(z_1-1)^2(z_2-1)^2}\right) = -\frac{1}{2}h_{11}.$$
Since $d_{w_0} = (2, 2)$, we get that in (\ref{int}) the residues at $w_0$ are different from zero only when $k = (0, 0)$. Thus, the total contribution at the root $w=w_0$ in the formula (\ref{int}) is represented by a polynomial of the fourth degree
\begin{equation*}
f_0(z) = \left(-2-(z_2-1)^2\right)\left(\frac{-1}{2}\right)\sum\limits_{\ell \in \{0, 1\}^2}c_{0, \ell}(z_1-1)^{\ell_1}(z_2-1)^{\ell_2}
\end{equation*}

For simple zeros, $w_j\in \boldsymbol{p}^{-1}(0)$ vector $d_{w_j} = (1, 1),$ therefore, the contributions they define to (\ref{int}) are given by one term for $\ell =k = (0, 0):$
$$f_j(z) = \det{H_{w_j}(z)}\cdot c_j\cdot\res\limits_{w_j}\left(\frac{1}{\boldsymbol{p}^I}\right)=c_j\frac{\det{H_{w_j}(z)}}{J_{\boldsymbol{p}}(w_j)}, \quad j = 1,2,3,4,$$
where $J_{\boldsymbol{p}}(w_j)$ is a value of the Jacobian of the mapping $\boldsymbol{p}$ at the point $w_j$. \linebreak For example, at the root $w_1 = (1+\sqrt{6}, 1+i\sqrt{2})$ determinant $\det{H_{w_1}(z)}$ \linebreak is equal to
\begin{multline*}
- z_{1} z_{2}^{3} + z_{1} z_{2}^{2}(3- \sqrt{2} i)+z_{1} z_{2}(- 3 + 2 \sqrt{2} i) + z_{1}(1 - \sqrt{2} i) +z_{2}^{3}(1- \sqrt{6}) +\\
z_{2}^{2}(3 \sqrt{6}-3 - 2 \sqrt{3} i+\sqrt{2} i) + z_{2}(3- 3 \sqrt{6} - 2 \sqrt{2} i+ 4 \sqrt{3} i) +(- 1 + \sqrt{6} - 2 \sqrt{3} i + \sqrt{2} i)
\end{multline*}

Calculations show that
$$\det{H_{w_1}(w_1)} = J_{p}(w_1) = 16 i\sqrt{3},$$
so contribution $f_1(z) = -i c_1[\det{H_{w_1}(z)}]/16\sqrt{3}.$

\section{Structure of local ideals $\langle\boldsymbol{p}\rangle_w$ under conditions (\ref{first}) and (\ref{second})}

Recall that the mapping $\boldsymbol{p} = (p_1, \ldots, p_n)$ at each point $w\in\boldsymbol{p}^{-1}(0)$ defines the ideal $\langle\boldsymbol{p}\rangle_w$ of the ring of germs $\mathcal{O}_w$ generated by $p_1, \ldots, p_n$. For brevity, we will assume that $w=0$.
\begin{lem}\label{lem}
If the conditions (\ref{first}), (\ref{second}) are met for the germs $p_1,\ldots, p_n\in \mathcal{O}_0$,
%
then in the ring $\mathcal{O}_0$ there is equality of ideals:
\begin{equation}
\langle p_1, \ldots, p_n \rangle = \langle z_1^{d_1}, \ldots, z_n^{d_n} \rangle.
\end{equation}

\end{lem}
\begin{proof}
Consider in $\mathbb{R}^n$ a simplex $S$ stretched over its vertices $v_j = (0, \ldots, d_j, \ldots, 0)$, $j =1,\ldots, n.$
From the condition (\ref{first}) it follows that the slices $p_i$ on $S$ constitute a set of initial weighted homogeneous polynomials for $p_i$:
$$(p_i)_{*} = m_{i 1}z_1^{d_1} + \ldots + m_{i n}z_n^{d_n}, \qquad i = 1, \ldots, n,$$
with weight $\delta = (\delta_1,\ldots, \delta_n)$, where $\delta_j = 1/d_j$. 
By condition (\ref{second}), the matrix $M = (m_{i j})$ is non-degenerate. By converting $M^{-1}$, the system $p_1, \ldots, p_n$ is reduced to the form $$P_i =z_i^{d_i}+g_i(z), \quad i =1, \ldots, n,$$
where each monomial $z^\beta$ of the germ $g_i$ has $\delta$--weighted degree of $\langle\beta, \delta\rangle$ is greater than that of the monome $z_i^{d_i}$. In this case , the equalities are fulfilled
\begin{equation}\label{34}
\frac{\partial^{\boldsymbol{\alpha}} P_i}{\partial z^{\boldsymbol{\alpha}}}(0) = 0
\text{\ \ for all \ \ } \boldsymbol{0}\leq \boldsymbol{\alpha} \leq \boldsymbol{d}-\boldsymbol{I}.
\end{equation}
Due to the non-degeneracy of $M$, there is equality of ideals
\begin{equation*}
\langle Q_1, \ldots, Q_n \rangle = \langle P_1, \ldots, P_n \rangle.
\end{equation*}
And from now on we will prove the equality of ideals
\begin{equation}\label{star}
\langle P_1, \ldots, P_n \rangle = \langle z_1^{d_1}, \ldots, z_n^{d_n} \rangle.
\end{equation}

It follows from the equalities (\ref{34}) that any monomial of $P_i$ is divisible by at least one element from the set $\{z_1^{d_1}, \ldots, z_n^{d_n}\}.$ Thus,
$$P_i= z_1^{d_1}\cdot f_{i 1}+z_2^{d_2}\cdot f_{i 2} +\ldots +z_n^{d_n}\cdot f_{i_n}.$$
Therefore, there is direct inclusion of ideals in (\ref{star}). 

Let's prove the inverse inclusion, that is, that each $z_i^{d_i}$ belongs to the left part (\ref{star}). To do this, we will use the local duality theorem from the theory of multidimensional residues (see \cite{Tsikh}, section 5.5). According to this theorem, $z_i^{d_i}\in \langle P_1, \ldots, P_n\rangle$ if and only if
$${\res\limits_0}\left(\frac{z_i^{d_i} h}{P^I}\right) = 0, \qquad \forall h \in \mathcal{O}_0.$$
Without limiting generality, we will conduct a proof for $i=1$ and first consider the case when $h = 1$:
\begin{equation}\label{resid}
{\res\limits_0}\left(\frac{z_1^{d_1}}{P^I}\right) = \frac{1}{(2\pi i)^n}\int\limits_\Gamma \frac{z_1^{d_1} dz}{P_1(z)\cdots P_n(z)},
\end{equation}
where $\Gamma = \Gamma_0 = \{z\in U_0: |P_j(z)| = \varepsilon_j, j = 1,\ldots,n \}.$

Note that outside the set of complex hypersurfaces \\
$\{P_1(z) = 0\}, \ldots, \{P_n(z) = 0\}$ the cycle $\Gamma_0$ is homologous to the torus
$$\gamma_0 = \{|z_1| = r^{\delta_1}, \ldots, |z_n| = r^{\delta_n} \}, \qquad r << 1.$$
In fact, since the monomials of germs $g_i$ have $\delta$--weighted degree greater than $z_i^{d_i}$, then $$\left.|z_i^{d_i}|\right|_{\gamma_0} =r>o_i(r) = \left.|g_i(z)|\right|_{\gamma_0},$$
where $o_i(r)$ are infinitesimal quantities relative to $r$. Therefore $P_i = z_i^{d_i}+ g_i \neq 0$ on $\gamma_0$, and according to (\cite{YA}, Lemma 4.9) for small $r$, the cycle $\Gamma_0$ is homologous to the cycle $\gamma_0$ in $U_0\backslash \{P_1 \cdot\ldots\cdot P_n = 0\}.$ Returning to the local residue (\ref{resid}), we will study it as an integral over the cycle $\gamma_0$:

\begin{equation*}
\begin{split}
 &\frac{1}{(2\pi i)^n}\int\limits_{\gamma_0} \frac{z_1^{d_1} dz}{(z_1^{d_1} + g_1(z))\cdots (z_n^{d_n} + g_n(z))} =\\
&= \frac{1}{(2\pi i)^n}\int\limits_{\gamma_0} \frac{z_1^{d_1} dz}{z_1^{d_1}\cdots z_n^{d_n}(1 + \frac{g_1(z)}{z_1^{d_1}})\cdots (1+ \frac{g_n(z)}{z_n^{d_n}})} = \\
&=\frac{1}{(2\pi i)^n}\int\limits_{\gamma_0} \sum\limits_{\boldsymbol{k}\geq \boldsymbol{0}}\frac{\left(-\frac{g_1(z)}{z_1^{d_1}}\right)^{k_1}\cdots \left(-\frac{g_n(z)}{z_n^{d_n}}\right)^{k_n} dz}{z_2^{d_2}\cdots z_n^{d_n}} = \\
&=\frac{1}{(2\pi i)^n} \sum\limits_{\boldsymbol{k}\geq \boldsymbol{0}} \int\limits_{\gamma_0} \frac{(-1)^{|\boldsymbol{k}|}g_1^{k_1}\cdots g_n^{k_n} dz}{z_1^{d_1 k_1} z_2^{d_2(k_2+1)}\cdots z_n^{d_n(k_n+1	)}}.
\end{split}
\end{equation*}
%
Here, the second equality is performed based on the geometric progression formula.
The last integrals give a nonzero contribution only for such monomials $z^\alpha$ in the numerator, for which
$$\alpha = (d_1 k_1 -1, d_2(k_2+1) -1, \ldots, d_n(k_n+1) -1).$$
A monomial from the multiplier $g_i^{k_i}$ can be obtained only from the product of $k_i$ monomials of the germ $g_i$. Therefore, we have $|k| := k_1+k_2+\ldots+k_n$ monoms of our germs, from which the numerator is obtained. But each monomial of the germ, as we showed above (based on the property (\ref{34})), is divided by at least one element from the set $\{z_1^{d_1}, \ldots, z_n^{d_n}\}.$ The monomials that are divisible by $z_1^{d_1}$ must be less than $k_1$, because otherwise in the numerator the power of $z_1$ will be greater than or equal to $d_1 k_1$, but we should get $d_1 k_1 -1$. Therefore, monomials that are divided into some element of the set $\{z_2^{d_2}, z_3^{d_3}, \ldots, z_n^{d_n}\}$, we have at least $k_2+k_3+\ldots+k_n+1$. Now suppose that for any $i \in \{2, \ldots, n\}$ the number of monomials divisible by $z_i^{d_i}$ is not more than $k_i$. In this case, the monomials that are divided into some element of the set $\{z_2^{d_2}, z_3^{d_3}, \ldots, z_n^{d_n}\}$, we have no more than $k_2+k_3+\ldots+k_n$. We get a contradiction, therefore there exists such $i$ that the monomials that are divisible by $z_i^{d_i}$ are greater than $k_i$. But in this case, in the numerator we will get a power of $z_i$ not less than $d_i(k_i+1)$. And this is more than the required degree $d_i(k_i+1)-1$. Thus, in any case, for some variable, the final degree will be greater than $-1$, and this will lead to the zeroing of the integral (\ref{resid}) corresponding to $h=1$.

Multiplying the numerator $z_1^{d_1}$ in the integral (\ref{resid}) by some irreversible germ $h$ will increase its order of zero at the point $z= 0$, so $${\res\limits_0}\left(\frac{z_1^{d_1}h}{P^I}\right) = 0, \qquad \forall h \in\mathcal{O}_0.$$
This is exactly what needed to be shown.

\end{proof}

\section{Proof of the theorem \ref{mainth}}

To prove the main theorem, we will need, along with the lemma \ref{lem}, the following statement.

\begin{lem}({\cite{Tsikh}, section 5.5}) \label{germs}
Let $q, g: G\to\mathbb{C}^n$ be two holomorphic maps of the domain $G \subset \mathbb{C}^n$ associated by the relation $g= Aq$ with the functional matrix $A =||a_{ij}(z)|/$, $a_{ij}\in\mathcal{O}(G).$ If for some point $w\in G$
$$q(w)\neq 0, g(w) = 0,$$ then $\det A$ belongs to the ideal $I_w(g)\subset\mathcal{O}_w$ generated in the ring of germs of holomorphic functions by the system $g=(g_1, \ldots, g_n).$
\end{lem}

\begin{proof}[Proof of the theorem \ref{mainth}]

If in the lemma \ref{germs} we put $g = p, q =(z-w)^{d_{w}}$, then we get that for $w\in p^{-1}(0)$ the determinant $ \det{H_w(z)}$ belongs to the local ideals $I_{w'}(p), w'\in p^{-1}(0)\backslash\{w\}$. Therefore, these ideals will also have an whole term equal to the product of $\det{H_w(z)}$ by the square bracket in (\ref{int}). Thus, this term will be vanished at points $w'$ under the action of the corresponding differential operators $\left\{\frac{\partial^{|\ell|}}{\partial z^\ell}\right\}_{\ell\in A_{w}}$ from the condition $\ref{first}$. In other words, this term will not give a contribution for the remaining points.

Now we show that it gives the necessary contribution at the point $z=w$. To do this, first learn how to calculate local residues $\res\limits_w\left(1/\boldsymbol{p}^I\right)$ relative to our mapping $\boldsymbol{p}$. The following local residue transformation formula will help us (see \cite{Tsikh}, section 5.5) when we are moving from the mapping $q$ to the mapping $g= Aq$: if $w$ is an isolated zero of the mappings $q$ and $g$, then for any germ $\varphi \in\mathcal{O}_w$

$${\res\limits_w}\left(\frac{\varphi}{q^I}\right) = {\res\limits_w}\left(\frac{\varphi \det A}{g^I} \right).$$

In our case, based on the decomposition (\ref{decomp}), we have $$q = p, g= ((z-w_1)^{d_{w_1}}, \ldots, (z-w_n)^{d_{w_n}}), A = {H_w^{-1}(z)}$$. Therefore, according to the transformation formula , we get:
$${\res\limits_{w}}\left( \frac{\varphi}{p^I} \right) = {\res\limits_{w}}\left(\frac{\varphi \det H_w^{-1}(z)}{g^I}\right) = {\res\limits_{w}}\left(\frac{\varphi}{g^I \det H_w(z)}\right).$$

Residue relative to $g$ is easily calculated: for a germ $\varphi = \sum\limits_{|\alpha|\geq 0} \varphi_\alpha(z-w)^\alpha$ it is equal to the coefficient $\varphi_\alpha$ with the index $\alpha = {(d_{w}-I)}$. Therefore, it is enough for us to decompose the germ $\frac{1}{\det H_w(z)}$ into a Taylor series centered at $z = w$:
$$\frac{1}{\det H_w(z)} = \sum\limits_{|\alpha|\geq0} h_{w,\alpha} (z_1-w_{1})^{\alpha_1}\ldots (z_n-w_{n})^{\alpha_n}.$$

Next, we show that for a fixed $\ell$ the multiplier for $\frac{c_{w,\ell}}{\ell!}(z_1-w_{1})^{\ell_1}\cdot (z_n-w_{n})^{\ell_n}$ in (\ref{int}) has the form
\begin{equation}\label{vid}
1+\text{ an element from the ideal } \langle\bold{(z-w)^{d_w-\boldsymbol{\ell}}}\rangle,
\end{equation}
where $\langle\bold{(z-w)^{d_w-\boldsymbol{\ell}}}\rangle$ is an ideal generated by the system $$(z_1-w_1)^{d_1-\ell_1}, \ldots, (z_n-w_n)^{d_n-\ell_n}.$$ This will just mean that this term will only give a contribution to the derivative of~a fixed order $\ell$, all other derivatives will vanish due to the residual term. Indeed, since $g^I = (z_1-w_1)^{d_1}\ldots(z_n-w_n)^{d_n}$, according to the Cauchy formula, the residue in (\ref{int}) is $h_{w,k}$, so this multiplier is equal to

\begin{multline*}
 \det {H_w(z)} \cdot
\sum_{k \leq d_{w} -I-\ell} (z-w)^k \res_{w}\left( \frac{(z-w)^{d_{w} -I-k}}{g^I \det{H_w(z)}}\right) =\\
 \det {H_w(z)} \cdot
  \sum_{k \leq d_{w} -I-\ell} h_{w, k} (z-w)^k  =  \\
 \det {H_w(z)} \cdot \left(\frac{1}{\det H_w(z)} + \text{ an element from the ideal } \langle\bold{(z-w)^{d_w-\boldsymbol{\ell}}}\rangle \right),
\end{multline*}
that is, it has the form (\ref{vid}).

When multiplying the remainder in (\ref{vid}) by our multiplier $$\frac{c_{w,\ell}}{\ell!}(z_1-w_{1})^{\ell_1}\cdots (z_n-w_{n})^{\ell_n}$$ we will get an element from the ideal $\langle\bold{(z-w)^{d_w-\boldsymbol{\ell}}}\rangle$. All derivatives of the orders ${\ell'} \in A_{w}\backslash \{\ell\}$ will be zeroed on this element at the point $w$.
The derivative of the order $\ell$ will be equal to $c_{w, \ell}$, which was required.

\end{proof}

\smallskip
{\bf Acknowlegement.} 
The investigation was supported by the Russian Science Foundation, grant No. 20-11-20117

\end{document}